\documentclass[11pt]{article}

\usepackage{amssymb, amsmath, amsthm, amsxtra, xspace}
\newtheorem{theorem}{Theorem}[section]
\newtheorem{lemma}{Lemma}[section]
\newtheorem{proposition}{Proposition}[section]

\newtheorem{corollary}{Corollary}[section]
\newtheorem{remark}{Remark}[section]

\renewcommand{\le}{\leqslant}

\renewcommand{\P}{\mathbb{P}\,}
\newcommand{\E}{\mathbb{E}\,}
\newcommand{\var}{\mbox{\rm Var}\,}
\newcommand{\Rd}{\mathbb{R}^{d}}
\newcommand{\R}{\mathbb{R}}
\newcommand{\N}{\mathbb{N}}

\newcommand{\Sd}{{\bf S}^{d-1}}

\newcommand{\eqd}{\stackrel{d}{=}}

\newcommand{\Vol}{\mbox{\rm Vol}}

\def\cH{{\mathcal H}}

\def\I{{\bf 1} }

\def\l{{\lambda}}

\begin{document}

\title{\bf Set Reconstruction by Voronoi cells}

\author{\bf{Matthias Reitzner\footnote{Postal address: Institut f\"ur Mathematik, Universit\"at Osnabr\"uck, 49069 Osnabr\"uck, Germany. 
Email: matthias.reitzner@uni-osnabrueck.de}, 
Evgeny Spodarev\footnote{Partially supported by RFBR-DFG (09-0191331) and DFG (436 RUS 113/962/0-1 R) grants.
\newline  Postal address: Institut f\"ur Stochastik, Universit\"at Ulm, 89069 Ulm, Germany. 
Email: evgeny.spodarev@uni-ulm.de}, 
Dmitry Zaporozhets\footnote{Partially supported by RFBR (10-01-00242), NSh-4472.2010.1, RFBR-DFG (09-0191331)  and DFG (436 RUS 113/962/0-1 R) grants. 
\newline  Postal address: St.Petersburg Department of V.A.Steklov Mathematical Institute, Fontanka 27, 191023 St.Petersburg, Russia. 
Email: zap1979@gmail.com} }}

\date{\today}

\newcommand{\Per}{\mathop{\rm Per}\nolimits}
\newcommand{\Div}{\mathop{\rm div}\nolimits}


\maketitle
\begin{abstract}
For a Borel set $A$  and a homogeneous Poisson point process  $\eta$  in $\R^d$ of intensity $\l >0$, define the
Poisson--Voronoi approximation  $ A_\eta$ of  $A$  as a union of all Voronoi cells with nuclei from $\eta$ lying in $A$. If $A$ has a finite volume and perimeter we find an exact asymptotic of $\E\Vol(A\Delta A_\eta)$ as $\lambda\to\infty$ where $\Vol$ is the Lebesgue measure. Estimates for all moments of $\Vol(A_\eta)$ and $\Vol(A\Delta A_\eta)$ together with their asymptotics for large $\l$ are  obtained as well.
\end{abstract}

\noindent
AMS Subject classification: Primary 60D05; secondary 60G55, 52A22, 60C05.
\newline \noindent
{\em Key words and phrases.} 
Poisson point process, Poisson-Voronoi cell, Poisson-Voronoi tessellation, perimeter

\section{Introduction}\label{sect:Intro}

Let $A$  be a  Borel set in $\R^d$ and $\eta$ be a Poisson point process in $\R^d$. Assume that we observe $\eta$ and the only information about $A$ at our disposal is which points of $\eta$ lie in $A$, i.e., we have the partition of the process $\eta$ into $\eta \cap A$ and $\eta\setminus A$. We try to reconstruct the set $A$ just by the information contained in these two point sets. For that we approximate $A$ by the set $ A_\eta$ of all points in $\R^d$ which are closer to $\eta \cap A$ than to $\eta\setminus A$.

More formally, let $\eta$ be a homogeneous Poisson point process
of intensity $\l >0$, and denote by $\upsilon_\eta (x) = \{ z\in\R^d:\ \|
z-x \| \leqslant \| z-y \| \mbox{ for all } y \in \eta  \} $ the
Voronoi cell generated by $\eta$ with nucleus $x\in\eta$. Then the
set $ A_\eta$ is just the union of the Poisson--Voronoi cells with
nuclei lying in $A$, i.e.,
$$ A_\eta = \bigcup\limits_{x \in \eta \cap A} \upsilon_\eta (x).$$
We call this set the {\it Poisson--Voronoi approximation } of the
set $A$. It was first introduced by Khmaladze and Toronjadze in~\cite{KhTo}. They proposed $A_\eta$ to be an estimator for $A$ when $\lambda$ is large (potential applications are listed in \cite[Section 1]{HR09}). They conjectured that for arbitrary bounded Borel set $A\subset\R^d, d\geqslant1,$ it holds
$$
\Vol(A_\eta)\to\Vol(A),\quad\lambda\to\infty,
$$
\begin{equation}\label{2047}
\Vol(A\Delta A_\eta)\to0,\quad\lambda\to\infty,
\end{equation}
almost surely, where $\Vol(\cdot)$ stands for the Lebesgue measure (volume) and $\Delta$ is the operation of the symmetric difference of sets. This conjecture was proved in~\cite{KhTo} for $d=1$. The case of general $d$ was treated by Einmahl and E. V. Khmaladze in \cite{EiKh} with some technical assumption on the boundary of $A$, and then generalized  by Penrose in~\cite{Pen07} to an arbitrary bounded Borel set $A$.

It can be easily shown (see Section~\ref{sec:tools} for details) that for any Borel set $A$ it holds
$$
 \E \Vol( A_\eta)= \Vol(A).
$$
Thus $\Vol( A_\eta)$ is an unbiased estimator for the
volume of $A$. In this paper we also consider  the $n$--th moment of $\Vol( A_\eta)$ and approximate it by the $n$--th degree of the volume of the original set $\Vol^n( A)$ asymptotically as $\lambda\to\infty$ (Theorem~\ref{thm:asympt_volApprox}). For the case when $n=2$ and $A$ is a convex compact, similar estimates were obtained in \cite{HR09}.

It might be suggested from \eqref{2047} that
\begin{equation}\label{2140}
\E\Vol(A\Delta A_\eta)\to0,\quad\lambda\to\infty,
\end{equation}
although it is not a direct corollary. The more interesting problem is to find an exact asymptotic of  $\E\Vol(A\Delta A_\eta)$. Initially it was considered by Heveling and Reitzner in \cite{HR09}. They proved that for any compact convex set $A$ with surface area $S(A)$ it holds
$$
\mathbb{E}\Vol(A \Delta A_\eta)=c_d\cdot S(A)\cdot \lambda^{-1/d}(1+O(\lambda^{-1/d})),\quad\lambda\to\infty,
$$
where the constant $c_d$ independent of $\lambda$ and $A$ was calculated by them in an explicit form (see Section~\ref{sect:Results} for details).
Here we obtain a similar asymptotic formula (Theorem~\ref{thm:MeanVol}) for a much wider class of sets.
Namely, we consider Borel sets with finite volume $\Vol (A)$ and  perimeter $\Per(A)$ (see Section~\ref{sec:tools} for the precise definition). Our methods are completely different from those of Heveling and Reitzner. The key observations are the connection between the Poisson--Voronoi approximation and the covariogram of $A$, and the  connection between the covariogram and the perimeter of a set recently established by Galerne \cite{Gal11}. As a by-product of our calculations, we prove that \eqref{2140} holds for any Borel set $A$ with finite volume (Corollary~\ref{cor:L1conv}).

We also consider higher moments of $\Vol(A \Delta A_\eta)$. For arbitrary Borel set $A$ we
approximate $\E\Vol^n(A \Delta A_\eta)$ by the $n$-th degree of $\E\Vol(A \Delta A_\eta)$ asymptotically
 as $\lambda\to\infty$ (Theorem~\ref{0037}). Thus, assuming that $\Vol (A), \Per(A)<\infty$ and using
the asymptotic for $\E\Vol(A \Delta A_\eta)$ from Theorem~\ref{thm:MeanVol}, we obtain the asymptotic for $\E\Vol^n(A \Delta A_\eta)$ (Corollary~\ref{cor:asympt_moments}).

\medskip
The paper is organized as follows. Our main results are stated in
the next section. In Section~\ref{sec:tools}, we introduce the
necessary background and notation, in particular the perimeter and
the covariogram of a set $A$. Proofs are given in Section
\ref{sect:Proofs}.

\section{Main results}\label{sect:Results}

Our first result yields the asymptotic of the average volume of $A\Delta A_\eta$ with increasing intensity $\lambda$. To formulate it, we need to define a notion of perimeter of a Borel set. The definition is somewhat technical, so we postpone it till Section~\ref{sec:tools}. If $A$ is a compact set with Lipschitz boundary (e.g. a convex body), then $\Per(A)$ equals the $(d-1)$-dimensional Hausdorff measure
 $\mathcal{H}_{d-1}(\partial A)$ of the boundary $\partial A$ of $A$. In general case it holds $\Per(A)\leqslant\mathcal{H}_{d-1}(\partial A)$
 (see, e.g. \cite[Proposition 3.62]{AF00}). Therefore, $\Per(A)$ could be replaced by $\mathcal{H}_{d-1}(\partial A)$ in the assumptions of the theorem.

\begin{theorem}\label{thm:MeanVol}
If $A \subset \R^d $ is a Borel set with $\Vol(A)<\infty$ and $\Per(A)<\infty$, then
\begin{equation}\label{2340}
\mathbb{E}\Vol(A \Delta A_\eta)=c_d\cdot \Per(A)\cdot \lambda^{-1/d}(1+o(1)),\quad\lambda\to\infty,
\end{equation}
where $ c_d=2d^{-2}\Gamma(1/d)\kappa_{d-1}\kappa_d^{{-1-1/d}}$ and
$\kappa_n$ is the volume of the unit $n$-dimensional ball.
\end{theorem}
The probabilistic intuition behind this asymptotic is the
following. The set difference $A \Delta A_\eta$ behaves
asymptotically as a very small tube neighbourhood  of the boundary
$\partial A$ formed out of the Poisson--Voronoi cells with nuclei
lying  almost on $\partial A$. Since the volume of a typical
Poisson--Voronoi cell is $\lambda^{-1}$, its diameter has the
order $\lambda^{-1/d}$, and so the volume of this tube
neighborhood has the order $\Per(A)\lambda^{-1/d}$.
\medskip

In the following, saying that some inequality holds asymptotically as $\lambda\to\infty$, we mean that it holds for
sufficiently large $\lambda\geqslant\lambda_0$. The choice of $\lambda_0$ might depend on $A$. Thus, all estimates are not uniform with
respect to $A$ (including those of Theorem~\ref{thm:MeanVol}).

\begin{theorem}\label{thm:asympt_volApprox}
If $A \subset \R^d $ is a Borel set with $\Vol(A)<\infty$, then
$$
\Big|\E \Vol^n(A_\eta)-\Vol^n(A)\Big|\leqslant C_{n,d}\cdot\Vol^{n-1}(A)\cdot\lambda^{-1},\quad\lambda\to\infty,
$$
where  $C_{n,d}$ is some constant  independent of $\lambda$ and $A$.
\end{theorem}
\begin{remark}
In fact, we show that the following non--asymptotic inequality holds: for any $\lambda>0$
$$
\Big|\E \Vol^n(A_\eta)-\Vol^n(A)\Big|\leqslant C_{n,d}\cdot\sum_{k=1}^{n-1}\Vol^{n-k}(A)\cdot\lambda^{-k}.
$$
\end{remark}
\begin{theorem}\label{0037}
If $A \subset \R^d $ is a Borel set with $\Vol(A)<\infty$ and $\Per(A)<\infty$, then
$$
\Big|\E\Vol^n(A\Delta A_\eta)-(\mathbb{E} \Vol(A\Delta A_\eta))^n\Big|\leqslant C^\prime_{n,d} \cdot \Per (A)^{n-1} \cdot\lambda^{-1-(n-1)/d},\quad\lambda\to\infty,
$$
where  $C^\prime_{n,d}$ is some constant  independent of $\lambda$ and $A$.
\end{theorem}

\begin{remark}
We conjecture that the following limit theorems can be proven by
the method of moments (see e.g. {\rm \cite[Theorems 30.1,
30.2]{Bill79}}):
\begin{equation} \label{eq:CLT}
 \lambda^{1/2(1+1/d)}\left( \Vol(A_\eta) - \Vol(A) \right) \to
N(0,\sigma_1 \Per(A)),
\end{equation}
$$
\lambda^{1/2(1+1/d)}\left( \Vol(A \Delta A_\eta) - c_d \Per(A)
\lambda^{-1/d} \right) \to N(0,\sigma_2 \Per(A) )
$$
in distribution as $\lambda\to\infty,$  $\sigma_1,  \sigma_2>0$.
\end{remark}
Recently (\ref{eq:CLT})  was proved by Schulte~\cite{Schulte2012} for {\it convex} sets $A$ using a central limit theorem for Wiener-It\'o chaos expansions. In his Remark~4 he points out that the result can be extended to all sets where the volume of a small tube neighbourhood $B(\partial A)$ of  $\partial A$ can be bounded in a nice way. Yet the general conjecture seems to be open.

\begin{corollary}\label{cor:asympt_moments}
If $A \subset \R^d $ is a Borel set with $\Vol(A)<\infty$ and $\Per(A)<\infty$, then 
$$
\E \Vol^n(A\Delta A_\eta)=\left( \mathbb{E} \Vol(A\Delta
A_\eta)\right)^n(1+O(\lambda^{-1+1/d})),\quad\lambda\to\infty,
$$
and for $d\geqslant2$
$$
\E \Vol^n(A\Delta A_\eta)=(c_d\Per(A))^n\lambda^{-n/d}(1+o(1)),\quad\lambda\to\infty.
$$
\end{corollary}

The asymptotic order of the variance of $A_\eta$ and $A \Delta
A_\eta$ as $\lambda\to\infty$ was first studied in \cite{HR09} for
convex sets $A$. We extend that results to arbitrary Borel sets.
\begin{corollary} \label{cor:var}
If $A \subset \R^d $ is a Borel set with $\Vol(A)<\infty$ and $\Per(A)<\infty$, then
$$
\var \Vol(A_\eta)\leqslant C_d\cdot  \Per(A)\cdot  \lambda^{-1- 1/d},\quad\lambda\to\infty,
$$
and
$$
\var \Vol(A \Delta A_\eta)\leqslant C_d\cdot  \Per(A)\cdot  \lambda^{-1- 1/d},\quad\lambda\to\infty,
$$
where  $C_{d}$ is some constant  independent of $\lambda$ and $A$.
\end{corollary}
The second inequality  follows immediately from  Theorem~\ref{0037}. The first inequality will be proved in Section~\ref{sect:Moments}.
\medskip

The probabilistic heuristic explaining the asymptotic behavior of
the variances is the following. Since $A \Delta A_\eta$ is
asymptotically a very small tube neighbourhood $B(\partial A)$ of
 $\partial A$ consisting of parts  $\tilde \upsilon_\eta(x)$ of  almost independent Poisson--Voronoi cells $\upsilon_\eta(x)$
with nuclei $x\in B(\partial A)$  we may use the formula for the
variance of the compound Poisson distribution:
$$
\var \Vol(A \Delta A_\eta) = \var \left( \sum_{x\in \eta \cap B(\partial A)}
\Vol(\tilde \upsilon_\eta(x)) \right)\approx  \var \left( \sum_{i=1}^N Y_i
 \right)
$$
where random variables $Y_i\eqd \Vol(\tilde \upsilon_\eta(x))$ are {\it
i.i.d.} and $$N\eqd {\rm card} (\eta \cap B(\partial A) )\sim
\mbox{Pois}(\l \Vol(B(\partial A)))$$ is independent of $Y_i$. Here
$\eqd$ means the equality in distribution and ${\rm card} (B)$ is
the cardinality of a set $B$. Then
\begin{multline*}
\var \left( \sum_{i=1}^N Y_i
 \right) =\E  N \,  \var Y_1 + \var N \, (\E Y_1)^2=\l
\Vol(B(\partial A))\, \E Y_1^2\\
  \le \lambda \Vol (B(\partial A)) \left( \E \Vol (\upsilon_\eta(x)) \right)^2  =O\left(\l \Per(A) \l^{-1/d}
\l^{-2}\right)=\Per(A) \, O\left(\l^{-1-1/d}\right)
\end{multline*} since $\tilde \upsilon_\eta(x)\subset \upsilon_\eta(x)$ for any $x$,  the
second moment of the volume of a typical Poisson--Voronoi cell is
of order $\lambda^{-2}$ and the volume of $B(\partial A)$ is of
order $\Per(A)\lambda^{-1/d}$.

\bigskip
The results of Corollary $\ref{cor:var}$ can also  be obtained by using
the  Poincar\'e inequality which gives
an upper bound on the variance of a functional of a
Poisson point process. Let $\cal N$ be the set of all locally finite configurations on $\R^d$. Consider a nonnegative measurable function $F\,:\, \cal N\to\R$. If $\E F^2(\eta) < \infty$, then
\begin{equation}\label{eq:EfronStein}
\var F(\eta )  \leqslant  \l \E  \int\limits_{\Rd} (  F(\eta \cup \{y\}
)- F(\eta))^2 \, dy,
\end{equation}
where we added a point $y$ to the Poisson point process $\eta$. Putting $F(\eta)=\Vol(A_\eta)$ in \eqref{eq:EfronStein}, we get
$$
\var \Vol(A_\eta) \leqslant \l \int\limits_{\Rd}  \E \left(
\Vol(A_{\eta\cup \{y\} }) - \Vol(A_\eta) \right)^2 \, dy,
$$
where the right--hand side can be estimated from above to get the
upper bound in Corollary $\ref{cor:var}$. The reasoning for the symmetric difference $A\Delta A_\eta$ is similar.

In full generality, inequality  \eqref{eq:EfronStein} was proved by Wu {\rm \cite{Wu}.}
As was shown by Last and Penrose  {\rm \cite[Theorem 1.2]{LastPenrose2010}}, it
is a consequence of an even more general inequality following from
the Fock space representation of Poisson point processes.

\section{Preliminaries}\label{sec:tools}
For basic facts from
integral geometry, stochastic geometry and Voronoi tessellations
which are not explained in the following, we refer the reader to
\cite{SchnWe3}, \cite{SKM95}, and \cite{Mol}.

Define the perimeter of a Borel set $A$ as
$$
\Per(A)=\sup\Big\{\int\limits_{A}\Div\varphi(x)\, dx\,:\,\varphi\in\mathcal{C}_c^1(\mathbb{R}^d),\|\varphi\|_\infty\leqslant1\Big\},
$$
cf. \cite{AF00}, where
$$
\Div\varphi(x)=\sum_{i=1}^d\frac{\partial\varphi_i}{\partial
x_i}\quad\mbox{and}\quad\|\varphi\|_\infty=\max_{i=1,\dots,d}\sup_{x\in\mathbb{R}^d}|\varphi_i(x)|
$$
for $\varphi=(\varphi_1,\dots,\varphi_d)$. The class
$\mathcal{C}_c^1(\mathbb{R}^d)$ consists of all continuously
differentiable vector--valued functions from $\mathbb{R}^d$ to
$\mathbb{R}^d$ with compact support.

Let $A$ be a Borel set with finite volume. Then
$$
g_A(x)=\Vol((A+x)\cap A),\quad x\in\R^d,
$$
is a covariogram of $A$. For the history on the covariogram
problem see the references in \cite{Gal11} and also the recent
breakthrough by Averkov and Bianchi \cite{AvBi09}.

In the proof of Theorem~\ref{thm:MeanVol} we use the result obtained by Galerne in~\cite[Theorem~14]{Gal11}. The following assertions are equivalent:
\begin{itemize}
\item[(a)] $\Per(A)<\infty$;
\item[(b)] there exists a finite limit
\begin{equation}\label{eq:1}
\lim_{r\to+0}\frac{g_A(ru)-g_A(0)}{r}  =  \frac{\partial g_A}{\partial u}(0)
\end{equation}
for all $u\in\mathbb{S}^{d-1}$;
\item[(c)] $g_A$ is Lipschitz.
\end{itemize}
In addition,  the Lipschitz constant of $g_A$ satisfies
\begin{equation}\label{0416}
\mbox{\rm Lip}(g_A)\leqslant\frac12\Per(A)
\end{equation}
and it holds
\begin{equation}\label{eq:2}
\int\limits_{\mathbb{S}^{d-1}}\frac{\partial g_A}{\partial u}(0) \, \cH_{d-1} (du) = - \kappa_{d-1} \Per(A).
\end{equation}

\bigskip

Another tool we need is the refined Campbell--Mecke formula for stationary point processes (cf. e.g. \cite{SKM95}). Using Slivnyak's theorem, we give its particular case for the Poisson point process.

As above, let $\eta$ be a homogeneous Poisson point process of intensity $\l >0$, and $\cal N$ be the set
of all locally finite point configurations on $\R^d$. Consider a nonnegative measurable function $f\,:\, {\cal N}\times(\R^d)^m\to\R$. Then
\begin{equation}\label{2352}
\E\sum_{(y_1,\dots,y_m)\in\eta^m_{\ne}}F(\eta,y_1,\dots,y_m)=
\lambda^m\int\limits_{(\R^d)^m}\E F(\eta\cup{\bar y_{m}},y_1,\dots,y_m)\,dy_1\dots dy_m,
\end{equation}
where $\eta^m_{\ne}$ denotes the set of all $m$--tuples of pair--wise distinct points from $\eta$,
and $\eta\cup{\bar y_{m}}$ is the process $\eta$ with added point set $\bar y_{m} =\{ y_1,\dots, y_m \}$.

As a simple corollary we get two identities which are crucial for us in the sequel.
\begin{proposition}
If $A \subset \R^d $ is a Borel set with $\Vol(A)<\infty$, then
\begin{equation}\label{1410}
\mathbb{E}\Vol(A_\eta)=\lambda\int\limits_{\mathbb{R}^d}\int\limits_{A}e^{-\lambda\kappa_d\|y-x\|^d}\,dy\,dx=\Vol(A),
\end{equation}
\begin{equation}\label{1411}
\mathbb{E}\Vol(A \Delta A_\eta)=
2\lambda \int\limits_{\mathbb{R}^d\setminus A} \int\limits_{A} e^{-\lambda\kappa_d\|y-x\|^d} \,dy \, dx.
\end{equation}
\end{proposition}
\begin{proof}
By Fubini's theorem and the Slivnyak-Mecke formula \eqref{2352}, we have
\begin{multline*}
\E\Vol(A_\eta)=
\E \int\limits_{\R^d}\I\left(x\in A_\eta\right)\,dx=\int\limits_{\R^d}\E\sum_{y\in\eta\cap A}\I\left(x\in\upsilon_{\eta}(y)\right)\,dx
\\=\lambda\int\limits_{\R^d} \int\limits_{A}\P\left(x\in\upsilon_{\eta\cup\{y\}}(y)\right) \,dy \,dx 
=\lambda\int\limits_{\R^d} \int\limits_{A}e^{-\lambda\kappa_d\|x-y\|^d} \,dy \,dx .
\end{multline*}
Similarly, we obtain
$$
\E\Vol(A\setminus A_\eta)=\lambda\int\limits_{A} \int\limits_{\R^d\setminus A}e^{-\lambda\kappa_d\|x-y\|^d}\,dy \,dx
$$
and
$$
\E\Vol(A_\eta\setminus A)=\lambda\int\limits_{\R^d\setminus A} \int\limits_{A}e^{-\lambda\kappa_d\|x-y\|^d}\,dy  \,dx .
$$

By definition $\Vol(A\Delta A_\eta)=\Vol(A\setminus A_\eta)+\Vol(A_\eta\setminus A)$ which completes the proof of \eqref{1411}. To prove the second part of \eqref{1410}, one has to apply Fubini's theorem and then use the formula
\begin{equation}\label{0609}
\int\limits_{\R^d}e^{-c\|x-y\|^d}\,dx =\frac{\kappa_d}c,\quad c>0,
\end{equation}
which could be easily proved by introducing spherical coordinates.
\end{proof}

Notice that we have  also proved that
\begin{equation*} 
\E\Vol(A\setminus A_\eta)=\E\Vol(A_\eta\setminus A).
\end{equation*}
However, $\Vol(A\setminus A_\eta)$ and $\Vol(A_\eta\setminus A)$ are not equidistributed since the first random variable is bounded, and the second is not.
As a direct colollary of the identity~(\ref{1410}) we get
\begin{equation}\label{1543}
\var\Vol( A_\eta)=\E\left( \Vol(A\setminus A_\eta)-\Vol(A_\eta\setminus A)\right)^2,
\end{equation}
which we shall use in the following.

\section{Proofs}\label{sect:Proofs}

\subsection{Asymptotics of the mean volume of the symmetric
difference }\label{sect:MeanVol}

In this section we give the proof of Theorem \ref{thm:MeanVol}. The key step to prove it is the
following relation between the Poisson--Voronoi approximation and
the covariogram of a set $A$.
\begin{lemma}\label{lemma:meanV_d}
Let  $g_A(x)$ be the covariogram of a Borel set $A$ with
$\Vol(A)<\infty$. Then
\begin{equation} \label{eq:symmDiff}
\mathbb{E}\Vol(A \Delta A_\eta) = -2 \int\limits_{0}^\infty r^{d-1}
e^{- \kappa_dr^d}\,
\tilde{g}_A( \l^{- 1/d} r ) \,  dr \, ,
\end{equation}
where
\begin{equation}\label{def:Sg}
\tilde{g}_A (r) = \int\limits_{\mathbb{S}^{d-1}} \left(g_A(ru) -
g_A(0)\right) \, \cH_{d-1} (du).
\end{equation}
\end{lemma}
\begin{proof}
Replacing $y$ in \eqref{1411} by $x-\l^{-1/d} z$ we get
\begin{eqnarray*}
\mathbb{E}\Vol(A \Delta A_\eta) &=&
2\lambda\int\limits_{\mathbb{R}^d}\int\limits_{\mathbb{R}^d}e^{-\lambda\kappa_d\|y-x\|^d}\mathbf{1}\{y\in
A,x\in A^c\} \, dy \, dx
\\ &=&
2 \int\limits_{\mathbb{R}^d}e^{- \kappa_d\|z\|^d}\int\limits_{\mathbb{R}^d}\mathbf{1}\{x\in(A+ \l^{-1/d} z)\cap A^c\} \,dz \,dx
\\ &=&
2 \int\limits_{\mathbb{R}^d}e^{- \kappa_d\|z\|^d} \Vol((A+\l^{-1/d} z)\cap A^c) \, dz.
\end{eqnarray*}
By the definition of the covariogram 
$ \Vol((A+ \l^{-1/d} z)\cap A^c)=g_A(0)-g_A(\l^{-1/d} z) $. 
We  introduce spherical coordinates  $z=ru$, where
$r\in\mathbb{R}^+$ and $u\in\mathbb{S}^{d-1}$. This yields
$$
\mathbb{E}\Vol(A \Delta A_\eta) = -2 \int\limits_{0}^\infty r^{d-1}
e^{- \kappa_dr^d}\, \Big[ \int\limits_{\Sd}\left( g_A(\l^{- 1/d}
ru) - g_A(0) \right) \, \cH_{d-1} (du) \Big]\, dr \, .
$$
\end{proof}

\begin{corollary}\label{cor:L1conv}
For any measurable $A$ with $\Vol(A)<\infty$ it holds
$$
\mathbb{E}\Vol(A\Delta A_\eta)\to 0,\quad\lambda\to\infty.
$$
\end{corollary}
\begin{proof}
It immediately follows from \eqref{eq:symmDiff} and the
continuity of the set covariogram.
\end{proof}

\begin{proof}[Proof of Theorem \ref{thm:MeanVol}]
Using Lemma \ref{lemma:meanV_d} and substituting  $t$ for  $\kappa_d r^d$ we obtain
$$
\mathbb{E}\Vol(A\Delta A_\eta) = -\frac2{d\kappa_d}
\int\limits_{0}^\infty e^{-t}\, \tilde{g}_A \left((\lambda\kappa_d)^{-1/d}
t^{1/d}\right) \, dt\,.
$$
It follows from \eqref{0416} and the definition of $\tilde{g}_A$  that
 \begin{equation}\label{eq:SgLip}
|\tilde{g}_A (r)| \leqslant \frac 12  \cH_{d-1} (\Sd) \Per(A) r .
\end{equation}
Therefore, Lebesgue's Dominated Convergence Theorem and equations \eqref{eq:1}, \eqref{eq:2} yield
\begin{multline*}
\lim_{\lambda\to\infty} \mathbb{E}\Vol(A\Delta A_\eta) \l^{1/d} 
= -\frac2 d \kappa_d^{-1-1/d} \lim_{\lambda\to\infty}
\int\limits_{0}^\infty e^{-t} t^{1/d} \, \frac {\tilde{g}_A
((\lambda\kappa_d)^{-1/d} t^{1/d})
}{(\lambda\kappa_d)^{-1/d} t^{1/d}} \, dt
\\= -\frac2 d \kappa_d^{-1-1/d}\int\limits_{0}^\infty e^{-t} t^{1/d}dt\int\limits_{\Sd} \frac{\partial g_A}{\partial u}(0)\, \cH_{d-1} (du)
\\=\frac2 d \kappa_{d-1} \kappa_d^{-1-1/d} \Per (A)
\int\limits_{0}^\infty e^{-t} t^{1/d} \,  dt
=\frac2 d \kappa_{d-1} \kappa_d^{-1-1/d}\Gamma\Big(1+\frac1d\Big) \Per (A).
\end{multline*}

\end{proof}

\subsection{Asymptotics of higher moments}  \label{sect:Moments}

To prove Theorem~\ref{thm:asympt_volApprox} and Theorem \ref{0037}, we need a number of lemmas. In this section $C$  is always some constant independent of $\lambda$ and $A$. Our first statement is the following version of H\"older's inequality.

\begin{lemma}\label{2301}
For any events $A_1,\dots,A_m$ it holds
$$
\P\left(\bigcap_{r=1}^mA_r\right)\leqslant\prod_{r=1}^m
\left(\P(A_r)\right)^{1/m}.
$$
\end{lemma}

\begin{lemma}\label{1534}
Let $x_0,y_0\in\R^d$. For any $\varepsilon>0$ and $m\in\N$ the
following inequality holds:
$$
\int\limits_{(\R^d)^m} \left( \P (x_0,x_1,\dots,
x_m\in\upsilon_{\eta\cup\{y_0\}}(y_0))\right)^\varepsilon
\, dx_1\dots dx_m \leqslant
e^{-\varepsilon\lambda\kappa_d\|x_0-y_0\|^d/(m+1)}\left(\frac{m+1}{\varepsilon\lambda}\right)^m.
$$
\end{lemma}
\begin{proof} By Lemma \ref{2301}, we have
\begin{multline*}
\int\limits_{(\R^d)^m}\left( \P (x_0,x_1,\dots,
x_m\in\upsilon_{\eta\cup\{y_0\}}(y_0))\right)^\varepsilon
\, dx_1\dots dx_m
\\
\leqslant\left( \P(x_0\in\upsilon_{\eta\cup\{y_0\}}(y_0)) \right)^{\varepsilon/(m+1)}
\int\limits_{(\R^d)^m} \prod_{i=1}^m\left( \P (x_i\in\upsilon_{\eta\cup\{y_0\}}(y_0)) \right)^{\varepsilon/(m+1)} 
\, dx_1\dots dx_m
\\
=e^{-\varepsilon\lambda\kappa_d\|x_0-y_0\|^d/(m+1)}
\left[\ \int\limits_{\R^d} \, e^{-\varepsilon \lambda\kappa_d \|x-y_0\|^d/(m+1)} \, dx \right]^m.
\end{multline*}
Using \eqref{0609} completes the proof.
\end{proof}
\begin{lemma}\label{0623}
For any $a>0$
\begin{equation}\label{0628}
\int\limits_{\mathbb{R}^d} \int\limits_{A}e^{-a\lambda\|y-x\|^d}\,dy \,dx = \frac{\kappa_d  \Vol(A)}{a \lambda}
\end{equation}
and
\begin{equation}\label{0629}
\int\limits_{\mathbb{R}^d\setminus A} \int\limits_{A}e^{-a\lambda\|y-x\|^d}\,dy \,dx 
\leqslant C\frac{\Per(A)}{\lambda^{1+1/d}},\quad \lambda\to\infty.
\end{equation}
\end{lemma}
\begin{proof}
The first equation follows from \eqref{1410} after replacing $\lambda$ by $\lambda' a/ \kappa_d$. The second estimate
follows from \eqref{1411} after replacing $\lambda$ by  $\lambda' a/ \kappa_d$ and then applying Theorem~\ref{thm:MeanVol}.
\end{proof}

Introduce the notation $B_r^x$ for the closed ball with center $x\in\Rd$ and radius $r>0$ in Euclidean metric.

\begin{lemma}\label{1654}
Let $x_1,x_2,y_1,y_2\in\R^d$. If $B^{x_1}_{\|x_1-y_1\|}\cap B^{x_2}_{\|x_2-y_2\|}\ne\emptyset$, then
$$
\P\left(B^{x_1}_{\|x_1-y_1\|}\cap\eta=\emptyset,B^{x_2}_{\|x_2-y_2\|}\cap\eta=\emptyset\right)
\leqslant2\exp\left(-\frac{\lambda\kappa_d}{2^{2d+1}}\left(\|x_1-y_2\|^d+\|x_2-y_1\|^d\right)\right).
$$
\end{lemma}
\begin{proof}
 Since $B^{x_1}_{\|x_1-y_1\|}\cap B^{x_2}_{\|x_2-y_2\|}\ne\emptyset$, it follows from the triangle inequality that
$$
\frac{\|x_1-y_2\|}{4},\frac{\|x_2-y_1\|}{4}\leqslant\max\left(\|x_1-y_1\|,\|x_2-y_2\|\right).
$$
Therefore, by Lemma \ref{2301} and stationarity of $\eta$ we have
\begin{multline*}
\P\left(B^{x_1}_{\|x_1-y_1\|}\cap\eta=\emptyset,B^{x_2}_{\|x_2-y_2\|}\cap\eta=\emptyset\right)
\\\leqslant\P\left(B^{x_1}_{\|x_1-y_2\|/4}\cap\eta=\emptyset,B^{x_1}_{\|x_2-y_1\|/4}\cap\eta=\emptyset\,\text{or}\,B^{x_2}_{\|x_1-y_2\|/4}\cap\eta=\emptyset,B^{x_2}_{\|x_2-y_1\|/4}\cap\eta=\emptyset\right)
\\\leqslant \sum_{i=1}^2 \P\left(B^{x_i}_{\|x_1-y_2\|/4}\cap\eta=\emptyset,B^{x_i}_{\|x_2-y_1\|/4}\cap\eta=\emptyset\right)
\\\leqslant \sum_{i=1}^2 \left( \P\left(B^{x_i}_{\|x_1-y_2\|/4}\cap\eta=\emptyset\right)\P\left(B^{x_i}_{\|x_2-y_1\|/4}\cap\eta=\emptyset\right)\right)^{1/2}
\\=2\exp\left(-\frac{\lambda\kappa_d}{2^{2d+1}}\left(\|x_1-y_2\|^d+\|x_2-y_1\|^d\right)\right).
\end{multline*}
\end{proof}

\begin{lemma}\label{1952}
For any $x_1,y_1,\dots,x_n,y_n\in\R^d$ it holds
\begin{multline*}
\P\left(B^{x_r}_{\|x_r-y_r\|}\cap\eta=\emptyset,r=1,\dots,n\right)\leqslant \exp\left(-\lambda\kappa_d\sum_{r=1}^n\|x_r-y_r\|^d\right)
\\+2\sum_{s<t}\exp\left(-\frac{\lambda\kappa_d}{n+1}\sum_{r=1}^n\|x_r-y_r\|^d\right)\exp\left(-\frac{\lambda\kappa_d}{2^{2d+1}(n+1)}\left(\|x_s-y_t\|^d+\|x_t-y_s\|^d\right)\right).
\end{multline*}
\end{lemma}
\begin{proof}
If the balls $B^{x_r}_{\|x_r-y_r\|},\,r=1,\dots,n$ are pairwise disjoint then we obviously have
$$
\P\left(B^{x_r}_{\|x_r-y_r\|}\cap\eta=\emptyset,r=1,\dots,n\right)=\exp\left(-\lambda\kappa_d\sum_{r=1}^n\|x_r-y_r\|^d\right).
$$
Suppose that for some indices $s\ne t$ it holds $B^{x_s}_{\|x_s-y_s\|}\cap B^{x_t}_{\|x_t-y_t\|}\ne\emptyset$. Applying Lemma~\ref{2301}, we get
\begin{multline*}
\P\left(B^{x_r}_{\|x_r-y_r\|}\cap\eta=\emptyset,r=1,\dots,n\right)
\\\leqslant   \left( \P\left(B^{x_s}_{\|x_s-y_s\|}\cap\eta=\emptyset,\,B^{x_t}_{\|x_t-y_t\|}\cap\eta=\emptyset\right)\right)^{1/(n+1)}  \prod_{r=1}^n \left( \P\left(B^{x_r}_{\|x_r-y_r\|}\cap\eta=\emptyset\right) \right)^{1/(n+1)}
\\=\exp\left(-\frac{\lambda\kappa_d}{n+1}\sum_{r=1}^n\|x_r-y_r\|^d\right)\left( \P\left(B^{x_s}_{\|x_s-y_s\|}\cap\eta=\emptyset,\,B^{x_t}_{\|x_t-y_t\|}\cap\eta=\emptyset\right)\right)^{1/(n+1)}.
\end{multline*}
It remains to apply  Lemma~\ref{1654} to finish the proof.
\end{proof}

\begin{proof}[Proof of Theorem~\ref{thm:asympt_volApprox}]

We have
\begin{multline}\label{0714}
\E \Vol^n(A_\eta) =
\E\int\limits_{(\R^d)^{n}} \I( \exists (y_1,\dots, y_n)\in (\eta\cap A)^n\,:\,x_i\in\upsilon_\eta(y_i), i=1,\dots n) \, dx_{1}\dots dx_{n}
\\=\sum_{i=1}^n\sum_{m_1+\dots+m_i=n}B_{n,i,m_1,\ldots,m_i}\beta_{i,m_1,\ldots,m_i},
\end{multline}
where
\begin{multline*}
\beta_{i,m_1,\ldots,m_i}=\int\limits_{(\R^d)^{n}}
\E \! \sum_{(y_1,\dots,y_i)\in(\eta\cap A)^i_{\ne}}  \!\! 
\I\left( x_1,\dots,x_{m_1}\in\upsilon_\eta(y_1), \dots ,x_{n-m_i+1},\dots,x_n\in\upsilon_\eta(y_i)\right)  
\\ \, dx_{1}\dots dx_{n}
\end{multline*}
and $B_{n,i,m_1,\ldots,m_i}$ denotes the number of ways to divide the set  $\{1,2,\dots,n\}$ into $i$ subsets of size $m_1,\ldots,m_i$. It it clear that
\begin{equation}\label{0942}
B_{n,n,1,\dots,1}=1.
\end{equation}

Fix some $i$ and $m_1,\dots,m_i$. Using the Slivnyak-Mecke formula \eqref{2352} we
get
\begin{multline*} 
\beta_{i,m_1,\ldots,m_i}=\lambda^{i}\int\limits_{(\R^d)^n}\int\limits_{A^i}
 \P \left( x_1,\dots,x_{m_1}\in\upsilon_{\eta\cup \tilde{y}_{i} }(y_1),\dots, x_{n-m_i+1},\dots,x_n\in\upsilon_{\eta\cup \tilde{y}_{i}}(y_i)\right)
\\ dy_1\dots dy_i \, dx_1\dots dx_n ,
\end{multline*}
where $\tilde{y}_{i}=\{ y_1,\ldots, y_{i} \}$. Taking into account that
$\upsilon_{\eta\cup \tilde{y}_{i} }(y_r) \subset \upsilon_{\eta\cup \{y_r\} }(y_r) $,
and using Fubini's theorem, Lemma~\ref{2301} and Lemma~\ref{1534}  we obtain 
\begin{eqnarray*}
\beta_{i,m_1,\ldots,m_i}
&\leqslant&
\lambda^{i}\int\limits_{A^i}  \prod_{r=1}^i\int\limits_{(\R^d)^{m_r}}
\left( \P\left( x_1,\dots,x_{m_r}\in\upsilon_{\eta\cup\{y_r\}}(y_r)  \right)\right)^{1/i}
\, dx_1\dots dx_{m_r} \, dy_1\dots dy_i
\\  &\leqslant&
\lambda^{i}\int\limits_{A^i} \prod_{r=1}^i   \Big( \frac{im_r}{ \l} \Big)^{m_r-1}
\int\limits_{\R^d}  \Big( e^{- \frac 1i \l \kappa_d \| x_1-y_r\|^d /m_r} \Big)  \, dx_1 \, dy_1\dots dy_i .
\end{eqnarray*}
By \eqref{0628} we get
$$
\beta_{i,m_1,\ldots,m_i}\leqslant
C\Vol^i(A)\lambda^{i-\sum_{r=1}^im_r}=C\Vol^i(A)\lambda^{i-n}.
$$

The maximum order of $\lambda$ is achieved for $i=n$,  which together with \eqref{0714} and \eqref{0942}  implies
\begin{multline*}
\E \Vol^n(A_\eta) 
\leqslant 
\lambda^{n} \int\limits_{(\R^d)^n} \int\limits_{A^n}
\P\left( x_r\in\upsilon_{\eta\cup \tilde{y}_{n}}(y_r),\; r=1,\dots,n\right)
\, dy_1\dots dy_n \, dx_1\dots dx_n
\\ +
C(\Vol(A))^{n-1}\lambda^{-1},\quad\lambda\to\infty.
\end{multline*}

It is clear that
$$
\P\left( x_r\in\upsilon_{\eta\cup \tilde{y}_{n}}(y_r),\; r=1,\dots,n\right)\leqslant\P\left(B^{x_r}_{\|x_r-y_r\|}\cap\eta=\emptyset,r=1,\dots,n\right).
$$
Therefore, by Lemma~\ref{1952},
\begin{equation}\label{1048}
\E \Vol^n(A_\eta) \leqslant v_n+2\sum_{s<t}v_{n,s,t}+C(\Vol(A))^{n-1}\lambda^{-1},\quad\lambda\to\infty,
\end{equation}
where
$$
v_n=\lambda^{n}\int\limits_{(\R^d)^n} \int\limits_{A^n} \exp\left(-\lambda\kappa_d\sum_{r=1}^n\|x_r-y_r\|^d\right)
\, dy_1\dots dy_n \, dx_1\dots dx_n,
$$
and
\begin{multline*}
v_{n,s,t}=\lambda^{n}\int\limits_{(\R^d)^n} \int\limits_{A^n} 
\exp\left(-\frac{\lambda\kappa_d}{n+1}\sum_{r=1}^n\|x_r-y_r\|^d\right)
\\ \times \exp\left(-\frac{\lambda\kappa_d}{2^{2d+1}(n+1)}\left(\|x_s-y_t\|^d+\|x_t-y_s\|^d\right)\right)
\, dy_1\dots dy_n \, dx_1\dots dx_n .
\end{multline*}
By formula~\eqref{1410},
\begin{equation}\label{1049}
v_n=\Vol^n(A).
\end{equation}
Let us estimate $v_{n,s,t}$. Using Fubini, it follows from \eqref{0628} that
\begin{eqnarray*}
v_{n,s,t}
&\leqslant&  C\Vol^{n-2}(A)\lambda^{2}
\int\limits_{\R^d}  \int\limits_{\R^d} \int\limits_{A} \int\limits_{A}
\exp\left(-\frac{\lambda\kappa_d}{(n+1)}\left(\|x_s-y_s\|^d+\|x_t-y_t\|^d\right) \right)
\\ &&
\hskip2cm \times\exp\left(-\frac{\lambda\kappa_d}{2^{2d+1}(n+1)}\left(\|x_s-y_t\|^d+\|x_t-y_s\|^d\right)\right)
\, dy_t  \, dy_s  \, dx_t  \, dx_s
\\
&\leqslant&  C\Vol^{n-2}(A)\lambda^{2}
\int\limits_{\R^d}  \int\limits_{A} 
\exp\left(-\frac{\lambda\kappa_d}{(n+1)}\left(\|x_s-y_s\|^d\right) \right)
\\ &&
\hskip1.2cm \times  \int\limits_{\R^d}  \int\limits_{\R^d} \exp\left(-\frac{\lambda\kappa_d}{2^{2d+1}(n+1)}\left(\|x_s-y_t\|^d+\|x_t-y_s\|^d\right)\right)
\, dy_t \, dx_t   \, dy_s   \, dx_s .
\end{eqnarray*}
Furthermore, by \eqref{0609},
$$
v_{n,s,t}\leqslant C\Vol^{n-2}(A)
\int\limits_{\R^d} \int\limits_{A} \exp\left(-\frac{\lambda\kappa_d}{(n+1)}\|x_s-y_s\|^d\right) \, dy_s   \, dx_s ,
$$
and applying \eqref{0628} again, we get
$$
v_{n,s,t}\leqslant C\Vol^{n-1}(A)\lambda^{-1}.
$$
Combining this with the estimate~\eqref{1048} and with \eqref{1049}, we
get
$$
\E \Vol^n(A_\eta) \leqslant \Vol^n(A)+C\Vol^{n-1}(A)\lambda^{-1}, \quad \lambda \to \infty.
$$
The application of Lyapunov's  inequality
$$
\E \Vol^n(A_\eta)\geqslant (\E\Vol(A_\eta))^n= \Vol^n(A)
$$
finishes the proof.

\end{proof}

\begin{proof}[Proof of Theorem~\ref{0037}]
We have
\begin{equation}\label{0032}
\E \Vol^n(A\Delta A_\eta)=\E\left( \Vol(A\setminus
A_\eta)+\Vol(A_\eta\setminus A)\right)^n
\\=\sum_{k=0}^n \binom n k u_k,
\end{equation}
where
\begin{equation}\label{1516}
u_k =  \E \int\limits_{A^{n-k}} \int\limits_{(\R^d\setminus A)^k} 
\I(x_1,\dots x_k\in  A_\eta,x_{k+1},\dots,x_n\not\in A_\eta)  
\, dx_1\dots dx_n.
\end{equation}
 Fix some $k$. We have  
\begin{multline}\label{0736}
u_k = \E  \int\limits_{A^{n-k}}  \int\limits_{(\R^d\setminus A)^k}
\I\Big( \exists (y_1,\dots, y_k)\in (\eta\cap A)^k, (y_{k+1},\dots, y_n)\in(\eta\setminus A)^{n-k}\,:
\\ \hfill x_i\in\upsilon_\eta(y_i), i=1,\dots n \Big) 
\, dx_1\dots dx_n
\\=\sum_{i=1}^k\sum_{j=1}^{n-k}\sum_{m_1+\dots+m_i=k}\sum_{l_1+\dots+l_j=n-k}B_{k,i,m_1,\ldots,m_i}B_{n-k,j,l_1,\ldots,l_j}\beta_{i,j,m_1,\ldots,m_i,l_1,\ldots,l_j},
\end{multline}
where
\begin{multline*}
 \beta_{i,j,m_1,\ldots,m_i,l_1,\ldots,l_j}
 =
 \int\limits_{A^{n-k}}   \int\limits_{(\R^d\setminus A)^k}
\E \sum_{(y_1,\dots,y_i)\in(\eta\cap A)^i_{\ne}}   \sum_{(y_{i+1},\dots,y_{i+j})\in(\eta\setminus A)^j_{\ne}}
\\     \hfill 
\I\left( x_1,\dots,x_{m_1}\in\upsilon_\eta(y_1),  \dots  ,
x_{n-l_j+1},\dots,x_{n}\in\upsilon_\eta(y_{i+j}) \right)
\\   \hfill
 dx_1\dots dx_n 
\end{multline*}
and $B_{k,i,m_1,\ldots,m_i},B_{n-k,j,l_1,\ldots,l_j}$ are the same combinatorial coefficients as in the proof of Theorem~\ref{thm:asympt_volApprox}.

Fix some $i,j$, and $m_1,\dots,m_i,l_1,\dots,l_j$. Using the Slivnyak-Mecke formula 
\eqref{2352} twice we get
\begin{multline*} 
\beta_{i,j,m_1,\ldots,m_i,l_1,\ldots,l_j}
=
\lambda^{i+j}  \int\limits_{A^{n-k}} \int\limits_{(\R^d\setminus A)^k} \int\limits_{(\R^d\setminus A)^{j}} \int\limits_{A^i} 
\\ \hfill
\P \left( x_1,\dots,x_{m_1}\in\upsilon_{\eta\cup \tilde{y}_{i+j} }(y_1),\dots,
x_{n-l_j+1},\dots,x_{n}\in\upsilon_{\eta\cup \tilde{y}_{i+j}}(y_{i+j}) \right)
\\   \hfill
 dy_1\dots dy_{i+j} \,   dx_1\dots dx_n ,
\end{multline*}
where $\tilde{y}_{i+j}=\{ y_1,\ldots, y_{i+j} \}$. 
By Fubini and Lemma~\ref{2301},
\begin{multline*}
\beta_{i,j,m_1,\ldots,m_i,l_1,\ldots,l_j}
\leqslant
\lambda^{i+j} \int\limits_{(\R^d\setminus A)^{j}} \int\limits_{A^i}
\\ \hfill
\prod_{r=1}^i\int\limits_{(\R^d\setminus A)^{m_r}}\left( \P\left( x_1,\dots,x_{m_r}\in\upsilon_{\eta\cup\{y_r\}}(y_r)
\right)\right)^{1/(i+j)} \, dx_1\dots dx_{m_r}
\\ \hfill \times\prod_{r=1}^j\int\limits_{A^{l_r}}\left( \P\left( x_1,\dots,x_{l_r}\in\upsilon_{\eta\cup\{y_{i+r}\}}(y_{i+r}) \right)\right)^{1/(i+j)}
\, dx_1\dots dx_{l_r}
\\   \hfill
dy_1\dots dy_{i+j} .
\end{multline*}
Using Lemma~\ref{1534} and \eqref{0629}, we get asymptotically as $\lambda\to\infty$
\begin{multline*}
\beta_{i,j,m_1,\ldots,m_i,l_1,\ldots,l_j}\leqslant
C\Per(A)^{i+j}
\lambda^{i+j+\sum_{r=1}^i(-m_r-1/d)+\sum_{r=1}^j(-l_r-1/d)} \\
=C\Per(A)^{i+j}\lambda^{-n+i+j-(i+j)/d}.
\end{multline*}
The maximum order of $\lambda$ is achieved for $i=k$, $j=n-k$, and the next term of maximum order is achieved for $i+j=n-1$,  which together with \eqref{0736} and \eqref{0942} implies
\begin{multline*}
u_k 
\leqslant 
\lambda^{n}  \int\limits_{A^{n-k}}\int\limits_{(\R^d\setminus A)^k}  \int\limits_{(\R^d\setminus A)^{n-k}} \int\limits_{A^k}
 \P\left( x_r\in\upsilon_{\eta\cup \tilde{y}_{n}}(y_r),\; r=1,\dots,n\right) \, dy_1\dots dy_n  \, dx_1\dots dx_n
\\ + C\Per(A)^{n-1}\lambda^{-1-(n-1)/d}
\end{multline*}
asymptotically as $\lambda\to\infty$. It is clear that
$$
\P\left( x_r\in\upsilon_{\eta\cup \tilde{y}_{n}}(y_r),\; r=1,\dots,n\right)\leqslant\P\left(B^{x_r}_{\|x_r-y_r\|}\cap\eta=\emptyset,r=1,\dots,n\right).
$$
Therefore, by Lemma~\ref{1952}, asymptotically as $\lambda\to\infty$,
\begin{equation}\label{2139}
u_k\leqslant v_k+2\sum_{s<t}v_{k,s,t}+C\Per(A)^{n-1}\lambda^{-1-(n-1)/d},
\end{equation}
where
\begin{multline*}
v_k
=
\lambda^{n} \int\limits_{A^{n-k}} \int\limits_{(\R^d\setminus A)^k}  \int\limits_{(\R^d\setminus A)^{n-k}} \int\limits_{A^k} 
 \exp\left(-\lambda\kappa_d\sum_{r=1}^n\|x_r-y_r\|^d\right)
\,  dy_{1}\dots dy_{n} \, dx_1 \dots dx_n ,
\end{multline*}
and
\begin{multline*}
v_{k,s,t}
=
\lambda^{n} \int\limits_{A^{n-k}} \int\limits_{(\R^d\setminus A)^k}  \int\limits_{(\R^d\setminus A)^{n-k}} \int\limits_{A^k} 
 \exp\left(-\frac{\lambda\kappa_d}{n+1}\sum_{r=1}^n\|x_r-y_r\|^d\right)
 \\ \times \exp\left(-\frac{\lambda\kappa_d}{2^{2d+1}(n+1)}\left(\|x_s-y_t\|^d+\|x_t-y_s\|^d\right)\right)
\,  dy_{1}\dots dy_{n} \, dx_1 \dots dx_n .
\end{multline*}
By the identity~\eqref{1411},
\begin{equation}\label{2138}
v_k=2^{-n}(\mathbb{E}\Vol(A \Delta A_\eta))^n.
\end{equation}
Let us estimate $v_{k,s,t}$. For instance, we assume that
$s\leqslant k$ and $t\geqslant k+1$ (other cases are treated in
the same way). In the same way as in the proof of Theorem~\ref{thm:asympt_volApprox}, we obtain by inequality~\eqref{0629}
\begin{eqnarray*}
v_{k,s,t}
&\leqslant &
C \Per(A)^{n-2}\lambda^{2-(n-2)/d}
\int\limits_{\R^d}   \int\limits_{\R^d\setminus A}  \int\limits_{\R^d} \int\limits_{A}  
\exp\left(-\frac{\lambda\kappa_d}{(n+1)}\left(\|x_s-y_s\|^d\right)\right)
\\ && \hskip2cm  \times 
\exp\left(-\frac{\lambda\kappa_d}{2^{2d+1}(n+1)}\left(\|x_s-y_t\|^d+\|x_t-y_s\|^d\right)\right)
\, dy_s\, dy_t \, dx_s  \,  dx_t
\end{eqnarray*}
as $\lambda \to \infty$. Furthermore, by~\eqref{0609},
$$
v_{k,s,t}\leqslant C\Per(A)^{n-2}\lambda^{-(n-2)/d}
\int\limits_{\R^d\setminus A} \int\limits_{A} 
\exp\left(-\frac{\lambda\kappa_d}{(n+1)}\|x_s-y_s\|^d\right)  \, dy_s dx_s,\quad \lambda\to\infty,
$$
and applying~\eqref{0629} again, we get
$$
v_{k,s,t}\leqslant C\Per(A)^{n-1}\lambda^{-1-(n-1)/d},\quad \lambda\to\infty.
$$
Combining this with \eqref{2139} and \eqref{2138}, we
get
\begin{equation}\label{1249}
u_k\leqslant2^{-n}(\mathbb{E}\Vol(A \Delta A_\eta))^n+C\Per(A)^{n-1}\lambda^{-1-(n-1)/d},\quad \lambda\to\infty.
\end{equation}
Inserting this into \eqref{0032} we otain 
$$
\E \Vol^n(A\Delta A_\eta)\leqslant (\mathbb{E}\Vol(A \Delta
A_\eta))^n+C\Per(A)^{n-1}\lambda^{-1-(n-1)/d},\quad \lambda\to\infty.
$$

The application of Lyapunov's  inequality
$$
\E \Vol^n(A\Delta A_\eta)\geqslant  \left(  \mathbb{E}\Vol(A\Delta
A_\eta)\right)^n
$$
finishes the proof.
\end{proof}

\begin{proof}[Proof of Corollary~\ref{cor:var}]
As was mentioned above, the second inequality immediately follows from  Theorem~\ref{0037}. To prove the first one, let us again combine~\eqref{1249} and \eqref{0032} now for $n=2$ and $k=0,1,2$. We get for sufficiently large $\lambda$
$$
\E\Vol^2(A\setminus A_\eta)+\E\Vol^2(A_\eta\setminus A)\leqslant\frac12(\mathbb{E}\Vol(A \Delta A_\eta))^2+2C\Per(A)\lambda^{-1-1/d},
$$
$$
2\E\left( \Vol(A\setminus A_\eta)\Vol(A_\eta\setminus A)\right)\leqslant\frac12(\mathbb{E}\Vol(A \Delta A_\eta))^2+2C\Per(A)\lambda^{-1-1/d}.
$$
Combining this with Lyapunov's  inequality
$$
\E\Vol^2(A\setminus A_\eta)+\E\Vol^2(A_\eta\setminus A)+2\E\left( \Vol(A\setminus A_\eta)\Vol(A_\eta\setminus A)\right)\\\geqslant(\mathbb{E}\Vol(A \Delta A_\eta))^2,
$$
we obtain for sufficiently large $\lambda$
$$
\E\Vol^2(A\setminus A_\eta)+\E\Vol^2(A_\eta\setminus A)-2\E\left( \Vol(A\setminus A_\eta)\Vol(A_\eta\setminus A)\right)\\\leqslant4C\Per(A)\lambda^{-1-1/d},
$$
which together with~\eqref{1543} completes the proof.
\end{proof}



\begin{thebibliography}{9}
\bibitem{AF00}
L. Ambrosio, N. Fusco, D. Pallara, \emph{Functions of bounded
variation and free discontinuity problems} (2000), Oxford
University Press.

\bibitem{AvBi09}
G. Averkov and G. Bianchi, \emph{Confirmation of Matheron's
conjecture on the covariogram of a planar convex body.},\ J. Eur.
Math. Soc. \textbf{11} (2009), 1187--1202.


\bibitem{Bill79}
P. Billingsley,  \emph{Probability and measure} (1979),\ J. Wiley
\& Sons, New York.


\bibitem{EiKh}
J. H. J. Einmahl and E. V. Khmaladze, \emph{The two-sample problem
in {$\R^m$} and measure-valued martingales.} In: \emph{State of
the art in probability and statistics ({L}eiden, 1999)}. IMS
Lecture Notes Monogr. Ser. \textbf{36} (2001), 434--463.

\bibitem{Gal11}
B. Galerne, \emph{Computation of the perimeter of measurable sets
via their covariogram. Applications to random sets.},\ Image Anal.
Stereol., \textbf{30} (2011), 39--51.


\bibitem{Hein10}
L. Heinrich, \emph{Asymptotic methods in statistics of random
point processes}, in Spodarev, E. (ed.), \emph{Lectures on
stochastic geometry, spatial statistics and random fields:
Asymptotic methods}. Lecture Notes in Mathematics (2012), Springer
(to appear).


\bibitem{HR09}
M. Heveling, M. Reitzner, \emph{Poisson-Voronoi approximation.}
Ann. Appl. Probab.,  \textbf{19} (2009), 719--736.

\bibitem{KhTo}
E. Khmaladze and N. Toronjadze, \emph{On the almost sure coverage
property of Voronoi tessellation: the ${\R}^ 1$ case}, Adv. Appl.
Probab. \textbf{33} (2001), 756--764.

\bibitem{LastPenrose2010}
G. Last and M. Penrose, \emph{Poisson process Fock space
representation, chaos expansion and covariance inequalities},
 Probab. Theory Relat. Fields \textbf{150}, No. 3--4 (2011), 663--690.


\bibitem{Mol}
J. M{\o}ller, \emph{Lectures on random Vorono\u\i\ tessellations},
Lecture Notes in Statistics \textbf{87} (1994), Springer, New
York.

\bibitem{Pen07}
M. D. Penrose, \emph{Laws of large numbers in stochastic geometry
with statistical applications}, Bernoulli \textbf{13} (2007),
1124--1150.

\bibitem{Schulte2012}
M. Schulte, \emph{A central limit theorem for the Poisson-Voronoi approximation}, manuscript, arXiv:1111.6466


\bibitem{SchnWe3}
R. Schneider and W. Weil, \emph{Stochastic and integral geometry}
 (2008), Springer, Berlin.

\bibitem{SKM95}
D. Stoyan, W. S. Kendall and J. Mecke, \emph{Stochastic geometry
and its applications}, 2nd ed. (1995), Wiley, Chichester.


\bibitem{Wu}
L. Wu, \emph{A new modified logarithmic Sobolev inequality for
Poisson point processes and several applications}, Probab. Theory
Relat. Fields \textbf{118} (2000), 427--438.

\end{thebibliography}
\end{document}